\documentclass[11pt]{amsart}
\usepackage{amsthm,amsmath,amssymb}

\usepackage{geometry}
\usepackage{graphicx,cite}
\usepackage{color}

\numberwithin{equation}{section}
\numberwithin{figure}{section}
\theoremstyle{plain}
\newtheorem{thm}{Theorem}[section]
\newtheorem{lem}[thm]{Lemma}

\newtheorem{claim}[thm]{Claim}

\theoremstyle{remark}
\newtheorem{rmk}[thm]{Remark}

\newcommand{\M}{\operatorname{M}}
\newcommand{\Hf}{\operatorname{\textbf{H}}}

\newcommand{\PP}{\operatorname{PP}}

\begin{document}

\title{A Shuffling Theorem  for Reflectively Symmetric Tilings}

\author{Tri Lai}
\address{Department of Mathematics, University of Nebraska -- Lincoln, Lincoln, NE 68588, U.S.A.}
\email{tlai3@unl.edu}
\thanks{This research was supported in part  by Simons Foundation Collaboration Grant (\# 585923).}




\date{\today}

\dedicatory{}

\begin{abstract}
In arXiv:1905.08311, the author and Rohatgi proved a shuffling theorem for doubly--dented hexagons. In particular, we showed that shuffling removed unit triangles along a horizontal axis in a hexagon only changes the tiling number by a simple multiplicative factor. In this paper, we consider a similar phenomenon for a symmetry class of tilings, the reflectively symmetric tilings, of the doubly--dented hexagons. We also prove several shuffling theorems for halved hexagons. These theorems generalize a number of known results in the enumeration of halved hexagons.
\end{abstract}

\maketitle

\section{Introduction}

MacMahon's classical theorem  \cite{Mac} on plane partitions fitting in a given box is equivalent to the fact that the number of lozenge tilings of a centrally symmetric hexagon $Hex(a,b,c)$ of side-lengths $a,b,c,a,b,c$ (in a cyclic order, from the north side) is given by the simple product:
\begin{equation}\label{Maceq}
\PP(a,b,c):=\prod_{i=1}^{a}\prod_{j=1}^{b}\prod_{k=1}^{c}\frac{i+j+k-1}{i+j+k-2}.
\end{equation}
The $10$ symmetry classes of the plane partitions were introduced in the classical paper of R. Stanley \cite{Stanley}. Each of the symmetry classes corresponds to a certain kind of symmetric lozenge tilings of a hexagon. We refer the reader to e.g.\cite{Andrews, Kup, Stem, Krat3, Zeil} and the lists of references therein for more related work about symmetric plane partitions .

 R. Proctor  \cite{Proc} enumerated a certain class of staircase plane partitions that are in bijection with the lozenge tilings of a hexagon with a maximal staircase cut off (see Figure \ref{halfhex6}(a)).

\begin{thm}[Proctor  \cite{Proc}] \label{Proctiling}For any non-negative integers $a,$ $b$, and $c$ with $a\leq b$, the number of lozenge tiling of the centrally symmetric hexagon $Hex(a,b,c)$, in which a maximal stair case is cut off from the west corner, is given by
\begin{equation}
\prod_{i=1}^{a}\left[\prod_{j=1}^{b-a+1}\frac{c+i+j-1}{i+j-1}\prod_{j=b-a+2}^{b-a+i}\frac{2c+i+j-1}{i+j-1}\right],
\end{equation}
 where empty products are taken to be 1.
\end{thm}

we denote by $\mathcal{P}_{a,b,c}$ the region in the above theorem and use the notation $\M(R)$ for the number of lozenge tilings of the region $R$ on the triangular lattice. When $a=b$, the region $\mathcal{P}_{a,b,c}$ above can be viewed as  a half of a symmetric hexagon with a zigzag cut along the vertical symmetry axis. In this point of view, we usually call the region $\mathcal{P}_{a,b,c}$ a \emph{halved hexagon (with defects)}.  We also note that when $a=b$, Proctor's Theorem \ref{Proctiling} implies an exact enumeration for one of the ten symmetry classes of plane partitions, the \emph{transposed-complementary plane partitions}.

\begin{figure}
  \centering
  \includegraphics[width=10cm]{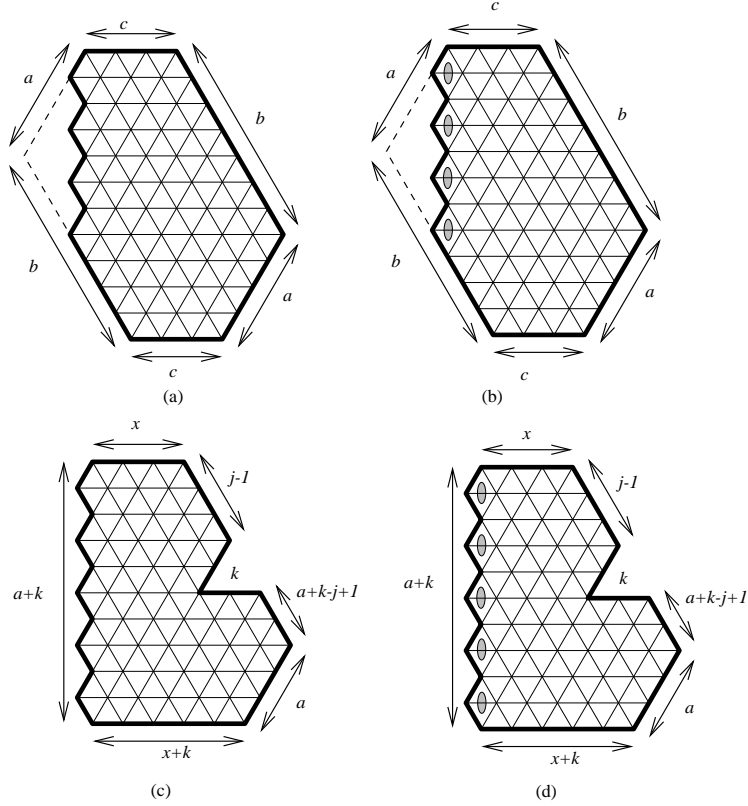}
  \caption{(a) The halved hexagon (with defects) $\mathcal{P}_{4,7,3}$. (b) The weighted halved hexagon $\mathcal{P}'_{4,7,3}$. (c) and (d) The regions in Rohatgi's paper \cite{Ranjan1}.}\label{halfhex6}
\end{figure}

Lozenges in a region can carry weights. In the weighted case, we use the notation $\M(R)$ for the sum of weights of all lozenge tilings of $R$, where the \emph{weight} of a lozenge tiling is the weight product of its constituent lozenges. We still call $\M(R)$ the \emph{(weighted) tiling number} of $R$. We consider the weighted counterpart $\mathcal{P}'_{a,b,c}$  of $\mathcal{P}_{a,b,c}$,  where all the vertical lozenges along the staircase cut are weighted by $\frac{1}{2}$ (see the vertical lozenges with shaded cores in Figure \ref{halfhex6}(b)).  In particular, each tiling of $\mathcal{P}'_{a,b,c}$ is weighted by $\frac{1}{2^n}$, where $n$ is the number of vertical lozenges running along the staircase cut. M. Ciucu \cite{Ciucu1} proved the following weighted version of Theorem \ref{Proctiling}.

\begin{thm} [Ciucu  \cite{Ciucu1}]\label{Ciucuhalfhex} For any non-negative integers $a,$ $b$, and $c$ with $a\leq b$
\begin{equation}
\M(\mathcal{P}'_{a,b,c})=2^{-a}\prod_{i=1}\frac{2c+b-a+i}{c+b-a+i}\prod_{i=1}^{a}\left[\prod_{j=1}^{b-a+1}\frac{c+i+j-1}{i+j-1}\prod_{j=b-a+2}^{b-a+i}\frac{2c+i+j-1}{i+j-1}\right].
\end{equation}
\end{thm}
We refer the reader to \cite{Cutoff} for a number of related tiling enumerations. It is worth noticing that R. Rohatgi \cite{Ranjan1} generalizes the regions $\mathcal{P}_{a,a,c}$ and $\mathcal{P}'_{a,a,c}$ to halved hexagons with a triangle removed from  the boundary.   The author latter generalized the tiling enumerations of halved hexagons by Proctor and by Rohatgi to halved hexagons in which one or two arrays of an arbitrary number of adjacent triangles have been removed from the boundary \cite{Halfhex1, Halfhex2, Halfhex3}.

MacMahon's tiling  formula (\ref{Maceq}) was generalized by Cohn, Larsen and Propp \cite[Proposition 2.1]{CLP} when they presented a correspondence between lozenge tilings of a semihexagon with unit triangles removed on the base and semi-strict Gelfand--Tsetlin patterns. In particular, the \emph{dented semihexagon} $T_{a,b}(s_1,s_2,\dots,s_a)$ is the region obtained from the upper half of the symmetric hexagon of side-lengths $b,a,a,b,a,a$ (in clockwise order, starting from the north side) by removing $a$ unit triangles along the base at the positions $s_1,s_2,\dotsc,s_a$ as they appear from left to right. The number of lozenge tilings of the dented semihexagon is given by
\begin{equation}\label{CLPeq}
\M(T_{a,b}(s_1,s_2,\dots,s_a))=\prod_{1\leq i<j \leq a}\frac{s_j-s_i}{j-i}.
\end{equation}
It is worth noticing that the author also proved simple product formulas for  several counterparts of Cohn--Larsen--Propp's theorem for the so-called `\emph{quartered hexagons}' in \cite{Quarter} (see the detailed statement in Lemma \ref{QAR} in the next section).

Recently, the author and Rohatgi proved a `\emph{shuffling theorem}' for lozenge tilings of a hybrid object between MacMahon's hexagon and Cohn--Larsen--Propp's, called \emph{doubly--dented hexagons} in \cite{shuffling}. A doubly--dented hexagon is a hexagon on the triangular lattice, like the hexagon in the case of MacMahon's theorem,  with an arbitrary set of unit triangles removed along a horizontal axis, like the `dents' in Cohn--Larsen--Propp's dented semihexagon (see Fig. \ref{multiplefernfig}).  In general, the tiling numbers of such regions are \emph{not} given by  simple product formula. However, we show that their tiling number only changes by a simple multiplicative factor when shuffling the positions of up- and down-pointing removed unit triangles. The main goal of this paper is to find a similar shuffling theorem for the reflectively symmetric tilings (i.e. the tilings which are invariant under refections over a vertical axis)  of the double--dented hexagons.

\medskip

For the completeness, we present in the next paragraph the shuffling theorem (Theorem 2.4 in \cite{shuffling}).

\medskip

Let $x,y,z,u,d$ be nonnegative integers, such that $u,d \leq n$. Consider a symmetric hexagon of side-lengths\footnote{From now on, we always list the side-lengths of a hexagon in the clockwise order from the north side.} $x+n-u,y+u,y+d,x+n-d,y+d,y+u$. We remove $n$ arbitrary unit triangles along the  horizontal lattice line $l$ that contains the west and the east vertices of the hexagon. Assume further that there are $u$ up-pointing removed unit triangles and $d$ down-pointing removed unit triangles. Denote $U=\{s_1,s_2,\dotsc,s_u\}$ and $D=\{t_1,t_2,\dots, t_d\}$ the sets of positions of the the up-pointing and down-pointing removed unit triangles (ordered from left to right), respectively (i.e., $U,D\subseteq [x+y+n]:=\{1,2,\dots,x+y+n\}$, $|U\cup D|=n$, and $U$ and $D$ are not necessarily disjoint). Assume that we have also a set of `barriers' at the positions in $B\subseteq [x+y+n]\setminus(U\cup D)$. Here a \emph{barrier} is a unit horizontal lattice interval that is not allowed to be contained in a lozenge of any tilings (see the red barriers in Fig.  \ref{multiplefernfig}; $B=\{6,13\}$ in this case). It means that we do not allow the appearance of vertical lozenges at the positions in $B$. Denote by $H_{x,y}(U;D;B)$ the hexagon with the above setup of removed unit triangles and barriers. We call it a \emph{doubly--dented hexagon (with barriers)}. We now allow to `shuffle' the positions of the up- and down-pointing unit triangles and `flip' those triangles (i.e. change their orientations from up-pointing to down-pointing, and vice versa) in the symmetric difference $U\Delta D$ to obtain new position sets $U'$ and $D'$, respectively. The following theorem shows that shuffling and flipping removed triangles only changes the tiling number  by a simple multiplicative factor. Moreover, the factor can be written in a similar form to Cohn--Larsen--Propp's formula (i.e. the product on the right-hand side of \ref{CLPeq}).

\begin{figure}\centering
\includegraphics[width=13cm]{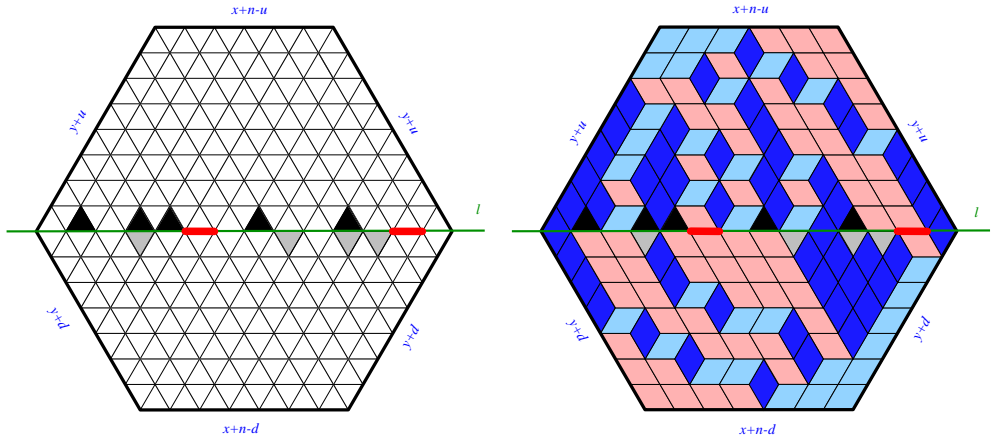}
\caption{The region $H_{4,3}(2,4,5,8,11;\ 4,9,11,12;\ 6,13)$ (left) and a lozenge tiling of its (right). The back and shaded triangles indicate the unit triangles removed; the bold horizontal bars indicate the barriers.}\label{multiplefernfig}
\end{figure}

\begin{thm}[Shuffling Theorem] For nonnegative integers $x,y,u,d,u',d',n$ ($u,d,u',d'\leq n$) and four ordered subsets $U=\{s_1,s_2,\dotsc,s_u\}$,  $D=\{t_1,t_2,\dots, t_d\}$, $U'=\{s'_1,s'_2,\dotsc,s'_{u'}\}$, and  $D'=\{t'_1,t'_2,\dots, t'_{d'}\}$ of $[x+y+n]$,  such that $U\cup D =U'\cup D'$ and $U\cap D =U'\cap D'$. Assume $B\subseteq[x+y+n]\setminus (U\cup D)$, such that $|B|\leq x$.
 We have
\begin{equation}\label{genmaineq}
  \frac{\M(H_{x,y}(U;D;B))}{\M(H_{x,y}(U';D';B))}= \frac{\displaystyle\prod_{1\leq i <j\leq u}\frac{s_j-s_i}{j-i}\displaystyle\prod_{1\leq i <j\leq d}\frac{t_j-t_i}{j-i}\PP(u,d,y)}{\displaystyle\prod_{1\leq i <j\leq u'}\frac{s'_j-s'_i}{j-i}\displaystyle\prod_{1\leq i <j\leq d'}\frac{t'_j-t'_i}{j-i}\PP(u'd',y)}.
\end{equation}
\end{thm}
An interesting aspect of the result is that the expression on the right-hand side does \emph{not} depend on the barrier set. Strictly speaking the above shuffling theorem only for symmetric hexagons, however, as mentioned in \cite[Remark 2.5]{shuffling}, this theorem implies a shuffling theorem for the general hexagon with unit triangle removed along an arbitrary horizontal axis.

\begin{figure}\centering
\includegraphics[width=13cm]{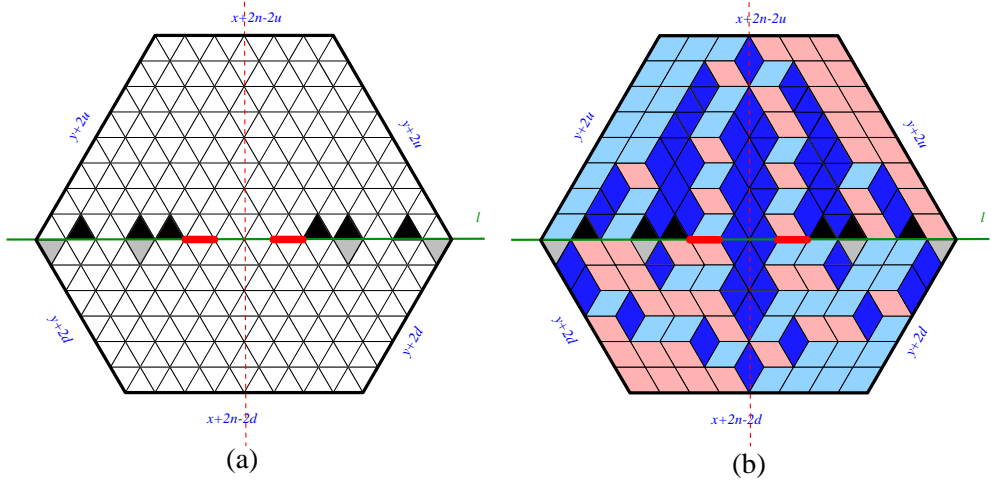}
\caption{(a) The region $RS_{4,2}(2,4,5;\ 1,4;\ 6)$ and (b) a reflectively symmetric tiling of its.}\label{Reflectdent}
\end{figure}

We now consider a reflectively symmetric doubly--dented hexagon as follows.  Let $U,D,B$ be three subsets of $\left\{1,2,\dots,\left\lceil \frac{x+y+2n}{2}\right\rceil\right \}$, such that $B\cap (U\cup D)=\emptyset$.  We denote by $RS_{x,y}(U;D;B)$ the symmetric doubly--dented hexagon
\[H_{x,y}(U\cup ((x+y+2n+1)-U);\  D\cup ((x+y+2n+1)-D);\ B\cup ((x+y+2n+1)-B)),\]
where, for an index set $S$ and a number $x$,  we denote $x-S:=\{x-s: s\in S\}$.
See Figure \ref{Reflectdent} for an example of the $RS$-type regions and reflectively symmetric tilings of them. It is easy to see that $x$ has to be even in order for the region to have reflectively symmetric tilings.

\begin{thm}[Shuffling Theorem for Reflectively Symmetric Tilings] \label{reflectfactor} For nonnegative integers $x,y,u,d,u',d',n$ ($x$ is even, $u,d,u',d'\leq n$) and five ordered subsets $U=\{s_1,s_2,\dotsc,s_u\}$,  $D=\{t_1,t_2,\dots, t_d\}$, $U'=\{s'_1,s'_2,\dotsc,s'_{u'}\}$, and  $D'=\{t'_1,t'_2,\dots, t'_{d'}\}$, and $B=\{k_1,\dots,k_b\}$ of $\left\{ 1,2,\dots,\left\lceil \frac{x+y+2n+1}{2}\right\rceil\right\}$,  such that $U\cup D =U'\cup D'$, $U\cap D =U'\cap D'$, $B\cap (U\cup D)=\emptyset$ and $b=|B|\leq x/2$.

(a) If $y$ is odd, then
\begin{align}\label{genmaineq2}
  &\frac{\M_r(RS_{x,y}(U;D;B))}{\M_r(RS_{x,y}(U';D';B))}= \frac{\displaystyle\prod_{1\leq i <j\leq u}(s^2_j-s^2_i)\displaystyle\prod_{1\leq i <j\leq d}(t^2_j-t^2_i)}{\displaystyle\prod_{1\leq i <j\leq u'}(s'^2_j-s'^2_i)\displaystyle\prod_{1\leq i <j\leq d'}(t'^2_j-t'^2_i)}  \frac{\Hf_2(2u'+y)\Hf_2(2d'+y)}{\Hf_2(2u+y)\Hf_2(2d+y)},
\end{align}

(b) If $y$ is even, then
\begin{align}\label{genmaineq3}
  \frac{\M_r(RS_{x,y}(U;D;B))}{\M_r(RS_{x,y}(U';D';B))}= &\frac{\displaystyle\prod_{1\leq i <j\leq u}(s_j-s_i)(s_j+s_i-1)}{\displaystyle\prod_{1\leq i <j\leq u'}(s'_j-s'_i)(s'_j+s'_i-1)}\notag\\
  &\qquad\times\frac{\displaystyle\prod_{1\leq i <j\leq d}(t_j-t_i)(t_j+t_i-1)}{\displaystyle\prod_{1\leq i <j\leq d'}(t'_j-t'_i)(t'_j+t'_i-1)}\frac{\Hf_2(2u'+y)\Hf_2(2d'+y)}{\Hf_2(2u+y)\Hf_2(2d+y)},
\end{align}
where we use the notation $\M_r(R)$ for the number of reflectively symmetric tilings of the region $R$, and where the ``\emph{skipping hyperfactorial}'' $\Hf_2(n)$ is defines as $\Hf_2(2k)=0!2!4!\cdots(2k-2)!$ and $\Hf_2(2k+1)=1!3!5!\cdots(2k-1)!$.
\end{thm}

An interesting aspect of the main result is that the tiling numbers on the left-hand side of each of identities (\ref{genmaineq2}) and (\ref{genmaineq3}) are \emph{not} given by simple product formulas in general. In other words, all large prime factors in the prime factorizations of the tiling numbers on the left-had side cancel out.

The rest of the paper is organized as follows. In Section 2, we present several fundamental results in enumerations of tilings and certain version of Kuo condensation \cite{Kuo} that will be employed in our proofs. Section 3 is devoted to shuffling theorems of four families of halved hexagons.
In Section 4, we present several asymptotic enumerations that are implied from the shuffling theorems in Section 3. Section 4  is devoted to several other applications of  the results in Section 3 in enumerating halved hexagons with arrays of triangles removed. The proofs of our main theorems will be presented in Section 5. We conclude the paper by a remark about the connection between to symplectic characters and interesting geometric interpretations of our shuffling theorems in Section 6.

\section{Kuo Condensation and other preliminary results}
A \emph{forced lozenge} in a region $R$ is a lozenge that appears in any tilings of $R$. If we removed forced lozenges $l_1,l_2,\dots,l_k$ from the region $R$ and obtain a new region $R'$, then the tiling number is changed by the reciprocal of weighted product of the forced lozenges, i.e.
\[\M(R')=\left(\prod_{i=1}^{k} wt(l_i)\right)^{-1} \M(R),\]
where $wt(l_i)$ denotes the weight of the lozenge $l_i$.

If a region admits a lozenge tiling, then it must have the same number of up-pointing and down-pointing unit triangles. We call such a region a \emph{balanced} region. The following lemma allows us to decompose a big region into smaller regions when enumerating tilings.

\begin{lem}[Region-splitting Lemma \cite{Tri1,Tri2}]\label{RS}
Let $R$ be a balanced region on the triangular lattice. Assume that a balanced sub-region $Q$ of $R$ satisfies the condition that the unit triangles in $Q$ that are adjacent to some unit triangle of $R-Q$ have the same orientation.
Then $\M(R)=\M(Q)\, \M(R-Q).$
\end{lem}

A \emph{perfect matching} (or simply \emph{matching} in this paper) of a graph is a collection of disjoint edges that covers all vertices of the graph.  A \emph{(planar) dual graph} of a region $R$ on the triangular lattice is the graph whose vertices are unit triangles in $R$ and whose edges connect precisely those two unit triangles sharing an edge. In the weighted case, the edges of the dual graph inherit the weights of the corresponding lozenges of the region.  The lozenge tilings of a region $R$ are in bijection with the matchings of its dual graph. In the view of this, we use the notation $\M(G)$ for weighted sum of matchings in $G$, where the weight of a matching of $G$ is the product of weights of its constituent edges. When $G$ is unweighted, i.e. all edges of $G$ have weight $1$, then $\M(G)$ counts matchings of $G$.

We will employ the following  version of `\emph{Kuo condensation}' introduced by Eric H. Kuo. We refer the reader to \cite{Kuo} for the other versions of Kuo condensation.
\begin{lem}[Kuo Condensation \cite{Kuo}]\label{kuothm}
Let $G=(V_1,V_2,E)$ be a (weighted) planar bipartite graph with the two vertex classes $V_1$ and $V_2$ such that $|V_1|=|V_2|+1$. Assume that $u,v,w,s$ are four vertices appearing in a cyclic order around a face of $G$ such that $u,v,w\in V_1$ and $s\in V_2$. Then
\begin{align}
\M(G-\{v\})&\M(G-\{u,w,s\})=\M(G-\{u\})\M(G-\{v,w,s\})+\M(G-\{w\})\M(G-\{u,v,s\}).
\end{align}
\end{lem}

We complete this section by  presenting the author's enumeration of `\emph{quartered hexagons}' as follows.

\medskip

We start with a trapezoidal region on the triangular lattice whose northern, northeastern, and southern have lengths $n,m,n+\left\lfloor\frac{m+1}{2}\right\rfloor$, respectively, and the western side follows the vertical  zigzag lattice path with $\frac{m}{2}$ steps (when $m$ is odd, the western side has $\frac{m-1}{2}$ and a half `bumps'). Next, we remove $k=\left\lfloor\frac{m+1}{2}\right\rfloor$ up-pointing unit triangles at the positions $a_1, a_2, \dotsc, a_k$ (ordered from left to right) from the base of the trapezoidal region and obtain the \emph{quartered hexagon} $L_{m,n}(a_1,a_2,\dotsc,a_k)$ (see Figure \ref{halfhex3b} (a)  for the case of even $m$, and Figure \ref{halfhex3b}(b) for the case of odd $m$). We also consider the weighted version $\overline{L}_{m,n}(a_1,a_2,\dotsc,a_k)$ of the quartered hexagon $L_{m,n}(a_1,a_2,\dotsc,a_k)$ by assigning to each vertical lozenge on the western  side a weight $\frac{1}{2}$ (see Figures \ref{halfhex3b}(c) and (d); the lozenges having shaded cores  are weighted by $\frac{1}{2}$). The author proved simple product formulas for the above families of quartered hexagons.

\begin{figure}
  \centering
  \includegraphics[width=8cm]{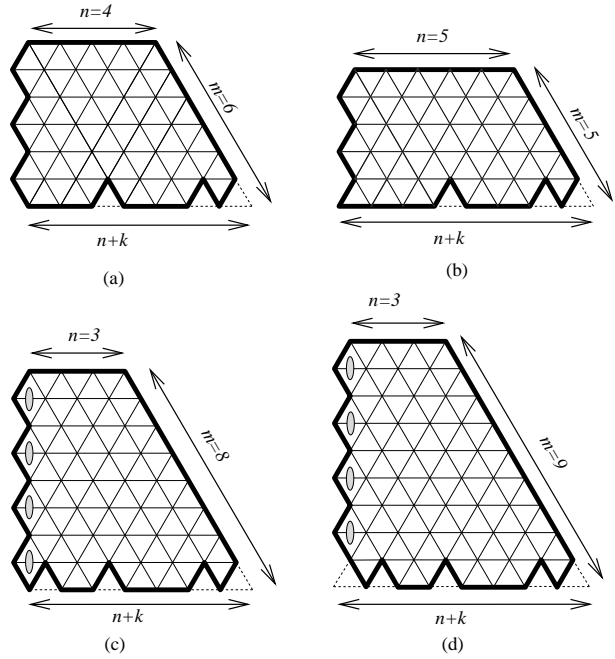}
  \caption{The quartered hexagons: (a) $L_{6,4}(3,6,7)$, (b)  $L_{5,5}(4,7,8)$, (c) $\overline{L}_{8,3}(1,3,6,7)$, and (d) $\overline{L}_{9,3}(1,2,4,7,8)$. The figure first appeared in \cite{Halfhex1}.}\label{halfhex3b}
\end{figure}

\begin{lem}[Lemma 3.1 in \cite{Halfhex1}]\label{QAR}
For any $1\leq k <n$ and $1\leq a_1<a_2<\dotsc<a_k\leq n$
\begin{equation}\label{(a)}
\M(L_{2k,n}(a_1,a_2,\dotsc,a_k))=\frac{a_1a_2\dotsc a_k}{\Hf_2(2k+1)}\prod_{1\leq i<j\leq k}(a_j-a_i)\prod_{1\leq i<j\leq k}(a_i+a_j),
\end{equation}
\begin{equation}\label{(b)}
\M(L_{2k-1,n}(a_1,a_2,\dotsc,a_k))=\frac{1}{\Hf_2(2k)}\prod_{1\leq i<j\leq k}(a_j-a_i)\prod_{1\leq i<j\leq k}(a_i+a_j-1),
\end{equation}
\begin{equation}\label{(c)}
\M(\overline{L}_{2k,n}(a_1,a_2,\dotsc,a_k))=\frac{2^{-k}}{\Hf_2(2k+1)}\prod_{1\leq i<j\leq k}(a_j-a_i)\prod_{1\leq i\leq j\leq k}(a_i+a_j-1),
\end{equation}
\begin{equation}\label{(d)}
\M(\overline{L}_{2k-1,n}(a_1,a_2,\dotsc,a_k))=\frac{1}{\Hf_2(2k)}\prod_{1\leq i<j\leq k}(a_j-a_i)\prod_{1\leq i<j\leq k}(a_i+a_j-2).
\end{equation}
\end{lem}

We note that the regions $L_{m,n}(a_1,a_2,\dotsc,a_{\lfloor\frac{m+1}{2}\rfloor})$ and $\overline{L}_{m,n}(a_1,a_2,\dotsc,a_{\lfloor\frac{m+1}{2}\rfloor})$ were first introduced in \cite{Quarter} by the first author when he enumerated domino tilings of the so-called `\emph{quartered Aztec rectangle}', a generalization of `\emph{quartered Aztec diamond}' introduced by Jockusch and Propp \cite{JP}.  We refer the reader to  e.g. \cite{Fischer, Tri3, Tri4, Tri5, KV} for more related work.

\section{Shuffling Theorem for Doubly-dented halved hexagons}

This section is devoted to shuffling theorems of four families of halved hexagons. We will use these shuffling theorems in our proof of the main theorem (Theorem \ref{reflectfactor}).

\begin{figure}\centering
\includegraphics[width=10cm]{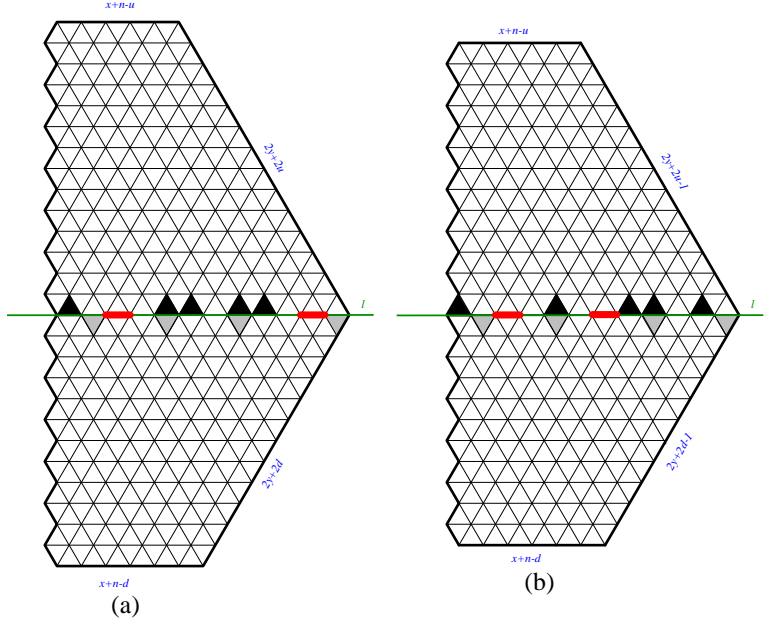}
\caption{Two halved hexagons: (a) $F_{3,2}(1,4,5,8,9;\ 2,5,8,12;\ 3,11)$ and (b) $\overline{F}_{3,2}(1,5,8,9,11;\ 2,5,9,12;\ 3,7)$.}\label{halvedhex}
\end{figure}

Consider a halved hexagon whose north, northeast, southeast, and south sides have lengths $x+n-u,2y+2u,2y+2d,x+n-d$, respectively, and the west side follows a vertical zigzag lattice path with $4y+2u+2d$ steps. Similar to the case of doubly--dented hexagons, we remove $u$ up-pointing unit triangles at the positions in the set $U$ and $d$ down-pointing unit triangles at the positions in the set $D$ along the horizontal lattice line $l$ containing the east vertex of the halved hexagon, such that $|U \cup D|=n$. We also place $b$ barriers at the positions in set $B\subseteq[x+y+n]\setminus(U\cup D)$. Denote by $F_{x,y}(U;D;B)$ the resulting region. See Fig.  \ref{halvedhex}(a) for an example.  We also have a variation $\overline{F}_{x,y}(U;D;B)$ of the region $F_{x,y}(U;D;B)$ illustrated in Fig.  \ref{halvedhex} (b). More precisely, the region $\overline{F}_{x,y}(U;D;B)$ is obtained from the halved hexagons of side-lengths $x+n-u,2y-1+2u,2y-1+2d,x+n-d,4y+2u+2d-2$ (in clockwise order, from the north side) by removing along the horizontal lattice line $l$ those up-pointing unit triangles at the positions in $U$ and down-pointing triangles at the positions in $D$, and placing barriers at the positions in $B$.

\begin{thm}[Shuffling Theorem for Halved Hexagons 1] \label{halffactor1} Assume that $x,y,u,d,n$ ($u,d,u',d'\leq n$) are nonnegative integers and that $U=\{s_1,s_2,\dotsc,s_u\}$,  $D=\{t_1,t_2,\dots, t_d\}$, $U'=\{s'_1,s'_2,\dotsc,s'_{u'}\}$, and  $D'=\{t'_1,t'_2,\dots, t'_{d'}\}$ are  four ordered subsets  of $[x+y+n]$,  such that $U\cup D =U'\cup D'$ and $U\cap D =U'\cap D'$. Assume $B\subset[x+y+n]\setminus (U\cup D)$, such that $|B|\leq x$, we have
\begin{align}\label{halfeq1}
  &\frac{\M(F_{x,y}(U;D;B))}{\M(F_{x,y}(U';D';B))}= \frac{\displaystyle\prod_{1\leq i <j\leq u}(s^2_j-s^2_i)\displaystyle\prod_{1\leq i <j\leq d}(t^2_j-t^2_i)}{\displaystyle\prod_{1\leq i <j\leq u'}(s'^2_j-s'^2_i)\displaystyle\prod_{1\leq i <j\leq d'}(t'^2_j-t'^2_i)}\frac{\Hf_2(2u'+2y+1)\Hf_2(2d'+2y+1)}{\Hf_2(2u+2y+1)\Hf_2(2d+2y+1)}
\end{align}
and
\begin{align}\label{halfeq2}
  \frac{\M(\overline{F}_{x,y}(U;D;B))}{\M(\overline{F}_{x,y}(U';D';B))}= &\frac{\displaystyle\prod_{1\leq i <j\leq u}(s_j-s_i)(s_j+s_i-1)}{\displaystyle\prod_{1\leq i <j\leq u'}(s'_j-s'_i)(s'_j+s'_i-1)}\notag\\
  &\times \frac{\displaystyle\prod_{1\leq i <j\leq d}(t_j-t_i)(t_j+t_i-1)}{\displaystyle\prod_{1\leq i <j\leq d'}(t'_j-t'_i)(t'_j+t'_i-1)}\frac{\Hf_2(2u'+2y)\Hf_2(2d'+2y)}{\Hf_2(2u+2y)\Hf_2(2d+2y)}.
\end{align}
\end{thm}

\begin{figure}\centering
\includegraphics[width=10cm]{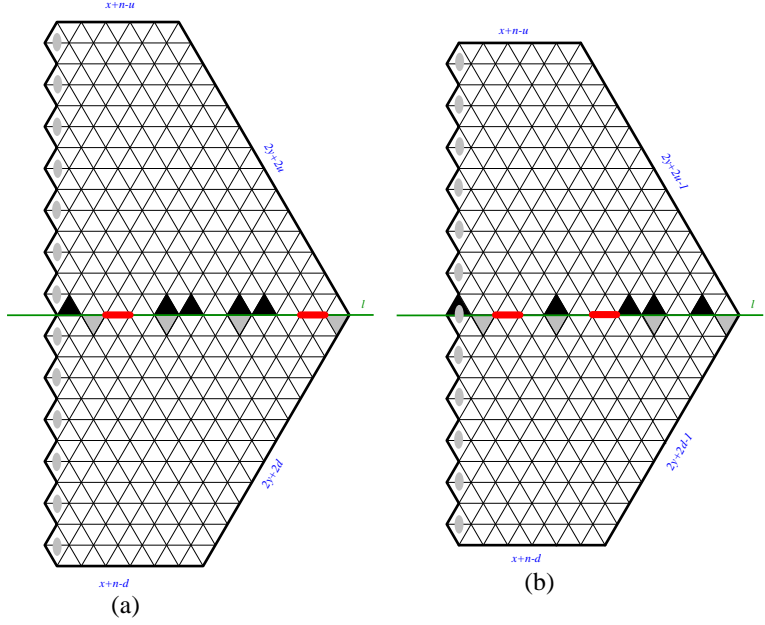}
\caption{Two weighted halved hexagons: (a) $W_{3,2}(1,4,5,8,9;\ 2,5,8,12;\ 3,11)$ and (b) $\overline{W}_{3,2}(1,5,8,9,11;\ 2,5,9,12;\ 3,7)$. The vertical lozenges with shaded cores are weighted by $\frac{1}{2}$.}\label{halvedhex4}
\end{figure}

Motivated by Ciucu's weighted counterpart of the halved hexagon $\mathcal{P}_{a,b,c}$ in \cite{Ciucu1}, we consider the following weighted versions of the $F$- and $\overline{F}$-type halved hexagons by assigning to each vertical lozenge on their west sides a weight $1/2$. Denote by $W_{x,y}(U;D;B)$ and $\overline{W}_{x,y}(U;D;B)$ the corresponding regions (see Figure \ref{halvedhex4} for examples).

\begin{thm} [Shuffling Theorem for Halved Hexagons 2] \label{halffactor2} With the same assumptions in Theorem \ref{halffactor1}, we have
\begin{align}\label{halfeq3}
  \frac{\M(W_{x,y}(U,D,B))}{\M(W_{x,y}(U',D',B))}= &\frac{\displaystyle\prod_{1\leq i <j\leq u}(s_j-s_i)\displaystyle\prod_{1\leq i \leq j\leq u}(s_i+s_j-1)}{\displaystyle\prod_{1\leq i <j\leq u'}(s'_j-s'_i)\displaystyle\prod_{1\leq i \leq j\leq u'}(s'_i+s'_j-1)}\notag\\
  &\times\frac{\displaystyle\prod_{1\leq i <j\leq d}(t_j-t_i)\displaystyle\prod_{1\leq i \leq j\leq d}(t_i+t_j-1)}{\displaystyle\prod_{1\leq i <j\leq d'}(t'_j-t'_i)\displaystyle\prod_{1\leq i \leq j\leq d'}(t'_i+t'_j-1)} \frac{\Hf_2(2u'+2y+1)\Hf_2(2d'+2y+1)}{\Hf_2(2u+2y+1)\Hf_2(2d+2y+1)}
\end{align}
and
\begin{align}\label{halfeq4}
  \frac{\M(\overline{W}_{x,y}(U,D,B))}{\M(\overline{W}_{x,y}(U',D',B))}= &\frac{\displaystyle\prod_{1\leq i <j\leq u}(s_j-s_i)(s_j+s_i-2)}{\displaystyle\prod_{1\leq i <j\leq u'}(s'_j-s'_i)(s'_j+s'_i-2)}\notag\\
  &\times\frac{\displaystyle\prod_{1\leq i <j\leq d}(t_j-t_i)(t_j+t_i-2)}{\displaystyle\prod_{1\leq i <j\leq d'}(t'_j-t'_i)(t'_j+t'_i-2)}\frac{\Hf_2(2u'+2y)\Hf_2(2d'+2y)}{\Hf_2(2u+2y)\Hf_2(2d+2y)}.
\end{align}
\end{thm}

We also note that, in general, the two tiling numbers on the left-hand side of each of the identities (\ref{halfeq1}), (\ref{halfeq2}), (\ref{halfeq3}), and (\ref{halfeq4}) are \emph{not} given by simple product formulas. In other words, all large prime factors in the prime factorizations of these tiling numbers cancel out.


\begin{rmk}
(1). One can also find similar shuffling theorems for `mixed-boundary halved hexagons' investigated by the author in  \cite{Halfhex3}. However, we only choose the above four families of halved hexagons to investigate in this paper.

(2). Similar to Remark 2.5 in \cite{shuffling}, one can obtain shuffling theorems for halved hexagons in which the removed unit triangle are running along an arbitrary horizontal axis (not necessarily on the axis containing the east vertex) by considering forced lozenges. The observation will not be consider in detail in this paper and left to the reader as an exercise.
\end{rmk}

\section{Asymptotic Results}
We call, in general, all removed unit triangles and barriers in our regions the `\emph{obstacles}'. We now assume that the set of all obstacles is partitioned into $k$ separated `clusters' (i.e. chain of contiguous unit triangles and barriers). Denote by $C_1,C_2,\dotsc,C_k$ these clusters and the distances between two consecutive ones are $d_1,d_2,\dotsc,d_{k-1}$ ($d_i>0$), ordered from left to right.  For the sake of convenience, we always assume that $C_1$ is attaching to the west vertex of the hexagon and that $C_k$ is attaching to the east vertex of the hexagon; $C_1,C_k$ may be empty.  We use the notation $F_{x,y}(C_1,\dots,C_k|\ d_1,\dots,d_{k-1})$ for our $F$-type regions in this section, and similarly for the $\overline{F}$-, $W$- and $\overline{W}$-type regions. Denote by $U_i,D_i,B_i$ the index sets of removed up-pointing unit triangles, removed down-pointing unit triangles, and barriers in the cluster $C_i$ (these index sets may be empty, $U_i$ and $D_i$ may be overlapping, and $B_i\cap (U_i\cup D_i)=\emptyset$). We `shuffle' the positions of up- and down-pointing triangles in the symmetric difference $U_i\Delta D_i$ of the cluster  $C_i$, i.e.  we can change the positions of the removed unit triangles in $U_i\Delta D_i$, but keep the orientations of those unit triangles stay put. This way, we get a new cluster $C'_i$ that has the same numbers of up- and down-pointing unit triangles as that in $C_i$. In the rest of this section we assume that $|C_i|=f_i$, $|U_i|=u_i$, and $|D_i|=d_i$ for any $i=1,2,\dots,k$.

 Motivated in part by the work of Ciucu and Krattenthaler in \cite{Ciu1,CK} and the author's generalizations in \cite{Threefern, shuffling}, we would like to investigate the behavior of the tiling number of large quartered hexagons with fixed clusters of obstacles.

\begin{thm}\label{dual1}For nonnegative integer $x,y$
\begin{align}\label{dualeq1}
\lim_{N\to \infty}&\frac{\M(F_{Nx,Ny}(C_1,\dots,C_k|\ Nd_1,\dots,Nd_{k-1}))}{\M(F_{Nx,Ny}(C'_1,\dots,C'_k|\ Nd_1,\dots,Nd_{k-1}))}=\prod_{i=1}^{k}\frac{s^+(C_i)s^-(C_i)}{s^+(C'_i)s^-(C'_i)}
\end{align}
 where $s^+(C_i)=\M(L_{2u_i,f_i-u_i}(U_i))$ and $s^-(C_i)=\M(L_{2d_i,f_i-d_i}(D_i))$ are the tiling numbers of the quartered hexagons whose dents are defined by the up-pointing triangles and down-pointing triangles in the cluster $C_i$, and where $s^+(C'_i)=\M(L_{2u_i,f_i-u_i}(U'_i))$ and $s^-(C'_i)=\M(L_{2u_i,f_i-u_i}(D'_i))$ are defined similarly w.r.t. $C'_i$.
\end{thm}
\begin{thm}\label{dual2}For nonnegative integer $x,y$
\begin{align}\label{dualeq2}
\lim_{N\to \infty}&\frac{\M(\overline{F}_{Nx,Ny}(C_1,\dots,C_k|\ Nd_1,\dots,Nd_{k-1}))}{\M(\overline{F}_{Nx,Ny}(C'_1,\dots,C'_k|\ Nd_1,\dots,Nd_{k-1}))}=\prod_{i=1}^{k}\frac{\overline{s}^+(C_i)\overline{s}^-(C_i)}{\overline{s}^+(C'_i)\overline{s}^-(C'_i)}
\end{align}
 where $\overline{s}^+(C_i)=\M(L_{2u_i-1,f_i-u_i}(U_i))$ and $\overline{s}^-(C_i)=\M(L_{2d_i-1,f_i-d_i}(D_i))$ are the tiling numbers of the quartered hexagons $L_{2k-1,n}$ whose dents are defined by the up-pointing triangles and down-pointing triangles in the cluster $C_i$, and where $s^+(C'_i)=\M(L_{2u_i-1,f_i-u_i}(U'_i))$ and $s^-(C'_i)=\M(L_{2u_i-1,f_i-u_i}(D'_i))$ are defined similarly w.r.t. $C'_i$.
\end{thm}
\begin{thm}\label{dual3}For nonnegative integer $x,y$
\begin{align}\label{dualeq3}
\lim_{N\to \infty}&\frac{\M(W_{Nx,Ny}(C_1,\dots,C_k|\ Nd_1,\dots,Nd_{k-1}))}{\M(W_{Nx,Ny}(C'_1,\dots,C'_k|\ Nd_1,\dots,Nd_{k-1}))}=\prod_{i=1}^{k}\frac{w^+(C_i)w^-(C_i)}{w^+(C'_i)w^-(C'_i)}
\end{align}
 where $w^+(C_i)=\M(\overline{L}_{2u_i,f_i-u_i}(U_i))$ and $w^-(C_i)=\M(\overline{L}_{2d_i,f_i-d_i}(D_i))$ are the tiling numbers of the quartered hexagons $\overline{L}_{2k,n}$ whose dents are defined by the up-pointing triangles and down-pointing triangles in the cluster $C_i$, and where $w^+(C'_i)=\M(\overline{L}_{2u_i,f_i-u_i}(U'_i))$ and $w^-(C'_i)=\M(\overline{L}_{2u_i,f_i-u_i}(D'_i))$ are defined similarly w.r.t. $C'_i$.
\end{thm}
\begin{thm}\label{dual4}For nonnegative integer $x,y$
\begin{align}\label{dualeq4}
\lim_{N\to \infty}&\frac{\M(\overline{W}_{Nx,Ny}(C_1,\dots,C_k|\ Nd_1,\dots,Nd_{k-1}))}{\M(\overline{W}_{Nx,Ny}(C'_1,\dots,C'_k|\ Nd_1,\dots,Nd_{k-1}))}=\prod_{i=1}^{k}\frac{\overline{w}^+(C_i)\overline{w}^-(C_i)}{\overline{w}^+(C'_i)\overline{w}^-(C'_i)}
\end{align}
 where $\overline{w}^+(C_i)=\M(\overline{L}_{2u_i-1,f_i-u_i}(U_i))$ and $\overline{w}^-(C_i)=\M(\overline{L}_{2d_i-1,f_i-d_i}(D_i))$ are the tiling numbers of the quartered hexagons $\overline{L}_{2k-1,n}$ whose dents are defined by the up-pointing triangles and down-pointing triangles in the cluster $C_i$, and where $\overline{w}^+(C'_i)=\M(\overline{L}_{2u_i-1,f_i-u_i}(U'_i))$ and $\overline{w}^-(C'_i)=\M(\overline{L}_{2u_i-1,f_i-u_i}(D'_i))$ are defined similarly w.r.t. $C'_i$.
\end{thm}

Theorem \ref{dual1} can be visualized as in Fig. \ref{geointer}, for $k=3$ (the quartered hexagons corresponding to $s^+(C_i)$ and $s^-(C_i)$ are the upper and lower halves of the `numerator hexagon' in the $i$-th fraction on the right-hand side; the  quartered hexagons corresponding to $s^+(C'_i)$ and $s^-(C'_i)$ are the upper and lower halves of the `denominator hexagon' in the $i$-th fraction, for $i=1,2,3$).
\begin{figure}\centering
\includegraphics[width=15cm]{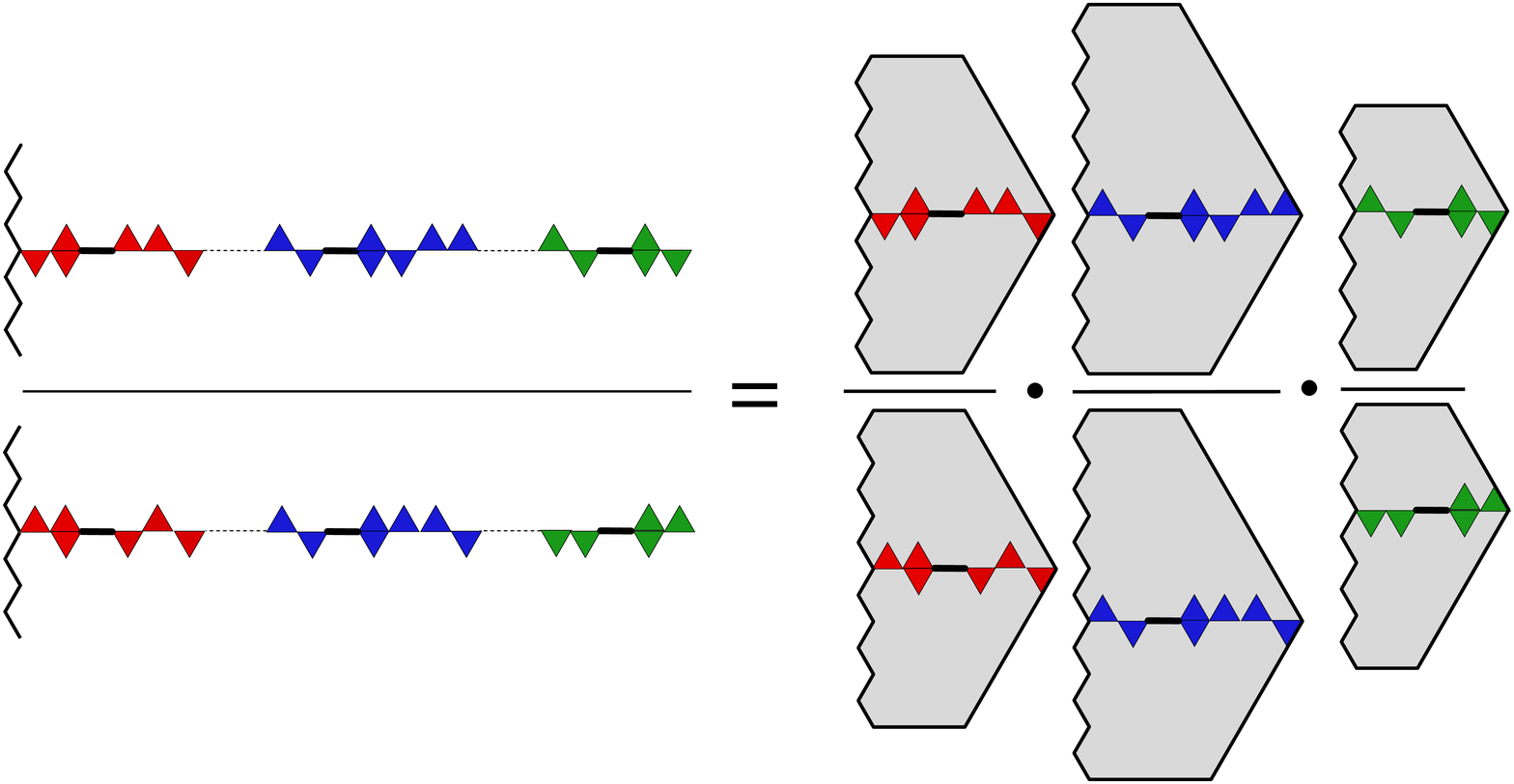}
\caption{Illustrating Theorem \ref{dual1}.}\label{geointer}
\end{figure}
We note that the first two asymptotic results of this type were introduced by Ciucu and Krattenthaler \cite{CK, Ciu1} as counterparts of MacMahon's classical theorem \cite{Mac}.

As the proof of the above four theorems are essentially the same, we only show here the proof of Theorem \ref{dual1}.
\begin{proof}
Assume further that the clusters $C_i$ and $C'_i$ both contain $u_i$ up-pointing triangles, $d_i$ down-pointing triangles, and $b_i$ barriers. Denote by $U_i=\{a^{(i)}_1,\dots,a^{(i)}_{u_i}\}$ the position set of up-pointing triangles in $C_i$ and $U'_i=\{e^{(i)}_1,\dots,e^{(i)}_{u_i}\}$ the corresponding position set in $C'_i$. Applying Theorem \ref{halffactor1} to the regions $F_{Nx,Ny}(U; D;B)$ and $F_{Nx,Ny}(U'; D';B)$, for $U:=\bigcup_{i=1}^{k}U_i$, $D:=\bigcup_{j=1}^{k}D_j$, $U':=\bigcup_{i=1}^{k}U'_i$,$D':=\bigcup_{j=1}^{k}D'_j$, $B=\bigcup_{i=1}^{k} B_i$, we get
\begin{equation}
\frac{\M(F_{Nx,Ny}(U;D;B))}{\M(F_{Nx,Ny}(U';D';B))}=\frac{\M(F_{Nx,Ny}(F_1,\dots,F_k|\ Nd_1,\dots,Nd_{k-1}))}{\M(F_{Nx,Ny}(F'_1,\dots,F'_k|\ Nd_1,\dots,Nd_{k-1}))}=\frac{\Delta(U)\Delta(D)}{\Delta(U')\Delta(D')},
\end{equation}
where, for an ordered set $S=\{0<s_1<s_2<\cdots<s_k\}$, the operator $\Delta$ is defined as $\Delta(S):=\prod_{1\leq i <j \leq k} (s_j^2-s_i^2)$. We observe that  $\Delta(S)^2=\prod_{1\leq i \not=j \leq k} |s_j^2-s_i^2|$.

We have
 \[\frac{\Delta(U)^2}{\Delta(U')^2}=\prod_{1\leq i,j\leq k}\prod_{p=1}^{u_i}\prod_{q=1}^{u_j}\frac{|(a^{(i)}_p)^2-(a^{(j)}_q)^2|}{|(e^{(i)}_p)^2-(e^{(j)}_q)^2|},\]
 where $p\not=q$ if $i=j$.
 We note that each  fraction $\frac{|(a^{(i)}_p)^2-(a^{(j)}_q)^2|}{|(e^{(i)}_p)^2-(e^{(j)}_q)^2|}$ on the right-hand side tends to 1 if $i\not=j$, as $N$ tends to the infinity. Thus, if we consider the limit, only the fractions of the form $\frac{|(a^{(i)}_p)^2-(a^{(i)}_q)^2|}{|(e^{(i)}_p)^2-(e^{(i)}_q)^2|}$ remain on the right-hand side.
Thus, $\frac{\Delta(U)^2}{\Delta(U')^2}$ tends to
\[\prod_{i=1}^{k}\prod_{1\leq p\not= q\leq u_i}\frac{|(a^{(i)}_p)^2-(a^{(i)}_q)^2|}{|(e^{(i)}_p)^2-(e^{(i)}_q)^2|}=\prod_{i=1}^{k}\frac{\Delta(U_i)^2}{\Delta(U'_i)^2}=\prod_{i=1}^{k}\frac{s^+(C_i)^2}{s^+(C'_i)^2}.\]
Thus, $\frac{\Delta(U)}{\Delta(U')}$ tends to $\prod_{i=1}^{k}\frac{s^+(C_i)}{s^+(C'_i)}$. Similarly, $\frac{\Delta(D)}{\Delta(D')}$ tends to $\prod_{i=1}^{k}\frac{s^-(C_i)}{s^-(C'_i)}$. This finishes the proof.
\end{proof}

\section{Applications to enumerations of halved hexagons with ferns removed}

As discussed in \cite{shuffling}, if we restrict ourself to the special case when the up-index set and down-index set are disjoint and there is no barrier, i.e. $U\cap D=\emptyset$ and $B=\emptyset$. The our doubly-dented hexagons are (tiling-)equinumerous with the hexagons with ferns removed.  This ways, our shuffling theorems for halved hexagons (Theorems \ref{halffactor1} and \ref{halffactor2}) can be related to enumerations of halved hexagons with `ferns' removed investigated by the author in \cite{Halfhex1, Halfhex2, Halfhex3}. A \emph{fern} is a chain of equilateral triangles of alternating orientations. We  also refer the reader to e.g. \cite{Ciu1,Twofern, Threefern, Threefern2} for further discussion of the fern structure.

There are four families of halved hexagons, $F$-, $\overline{F}$-,$W$ and $\overline{W}$-types,  to investigate. However, we only focus on the $F$-type halved hexagons in this section. The investigation of other three is left to the reader.

Assume that the set of removed unit triangles in the region $F_{x,y}(U;D;\emptyset)$ is partitioned into $k$ clusters $C_1,C_2,\dots,C_k$, in which the distances between two consecutive ones are $d_1,d_2,\dots, d_{k-1}$ from left to right. Similar  to the previous section, we use the notation $F_{x,y}(C_1,\dots,C_k |\ d_1,\dots,d_{k-1})$ for our region, the only difference is that $C_i$ now contains no barrier and $U_i\cap D_i=\emptyset$, for $i=1,2,\dots,k$.

Since $U\cap D=\emptyset$, each cluster $C_i$ is partitioned further to maximal  intervals of unit triangles of the same orientation. We call these interval \emph{up-intervals} or \emph{down-intervals} based on the orientation of their unit triangles.  By removing forced lozenges along each up- and down-interval in the cluster $C_i$ if it contains at least 2 triangles, we get a fern  $F_i$. Denote by $Q_{x,y}(F_1,\dots,F_k|\ d_1,\dots,d_{k-1})$ the resulting hexagon with ferns removed  (see Figure \ref{reflectforced}; the forced lozenges are the vertical white ones and the $Q$-type region is the shaded one).

\begin{figure}\centering
\includegraphics[width=8cm]{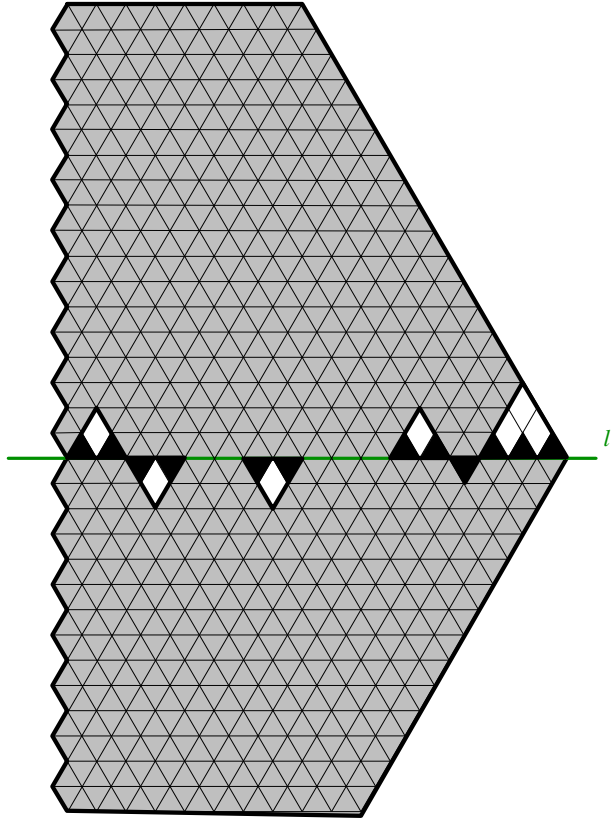}
\caption{Obtaining a halved hexagon with ferns removed $Q_{x,y}(F_1,\dots,F_k|\ d_1,\dots,d_{k-1})$ from the region $F_{x,y}(C_1,\dots,C_k |\ d_1,\dots,d_{k-1})$.}\label{reflectforced}
\end{figure}

\begin{figure}\centering
\setlength{\unitlength}{3947sp}%
\begingroup\makeatletter\ifx\SetFigFont\undefined%
\gdef\SetFigFont#1#2#3#4#5{%
  \reset@font\fontsize{#1}{#2pt}%
  \fontfamily{#3}\fontseries{#4}\fontshape{#5}%
  \selectfont}%
\fi\endgroup%
\resizebox{15cm}{!}{
\begin{picture}(0,0)%
\includegraphics{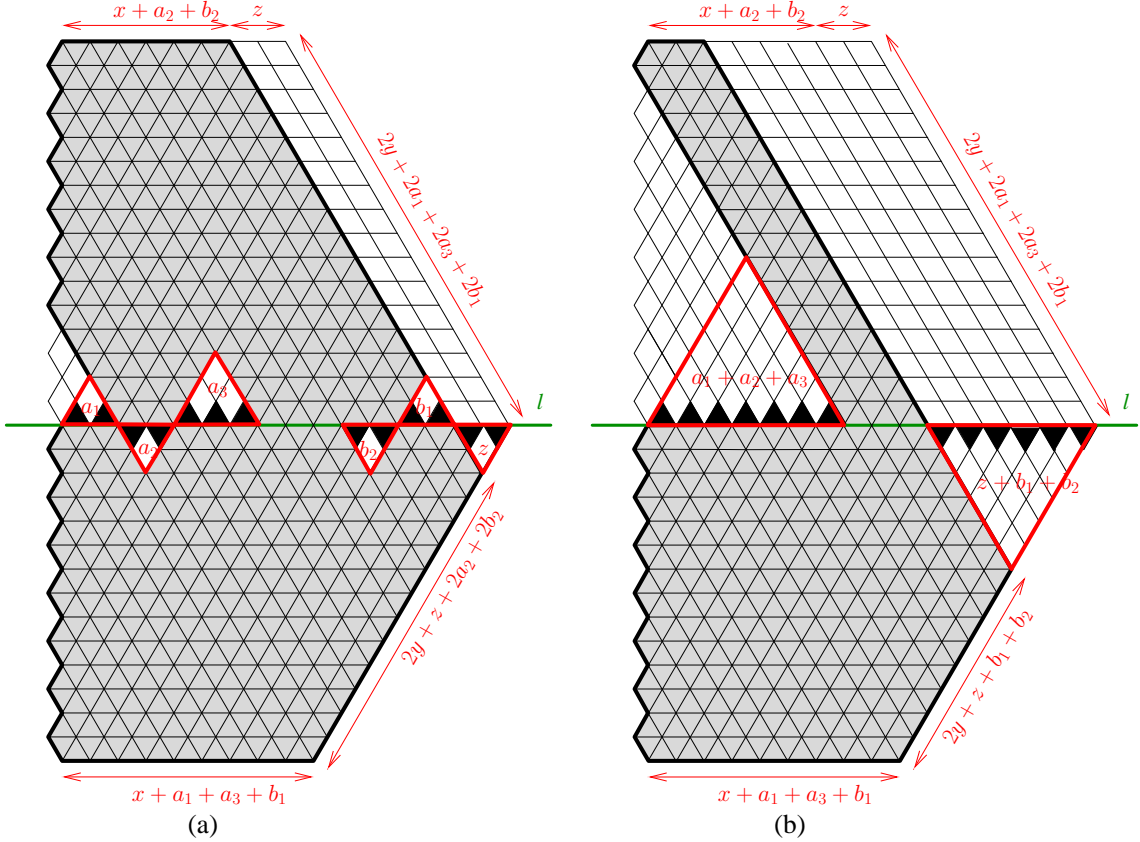}%
\end{picture}%
%
%

\begin{picture}(16651,12559)(980,-16778)
\put(13205,-4574){\makebox(0,0)[lb]{\smash{{\SetFigFont{20}{24.0}{\rmdefault}{\mddefault}{\itdefault}{\color[rgb]{1,0,0}$z$}%
}}}}
\put(11450,-16170){\makebox(0,0)[lb]{\smash{{\SetFigFont{20}{24.0}{\rmdefault}{\mddefault}{\itdefault}{\color[rgb]{1,0,0}$x+a_1+a_3+b_1$}%
}}}}
\put(14914,-15342){\rotatebox{60.0}{\makebox(0,0)[lb]{\smash{{\SetFigFont{20}{24.0}{\rmdefault}{\mddefault}{\itdefault}{\color[rgb]{1,0,0}$2y+z+b_1+b_2$}%
}}}}}
\put(7929,-11056){\makebox(0,0)[lb]{\smash{{\SetFigFont{20}{24.0}{\rmdefault}{\mddefault}{\itdefault}{\color[rgb]{1,0,0}$z$}%
}}}}
\put(7021,-10449){\makebox(0,0)[lb]{\smash{{\SetFigFont{20}{24.0}{\rmdefault}{\mddefault}{\itdefault}{\color[rgb]{1,0,0}$b_1$}%
}}}}
\put(6196,-11086){\makebox(0,0)[lb]{\smash{{\SetFigFont{20}{24.0}{\rmdefault}{\mddefault}{\itdefault}{\color[rgb]{1,0,0}$b_2$}%
}}}}
\put(3969,-10149){\makebox(0,0)[lb]{\smash{{\SetFigFont{20}{24.0}{\rmdefault}{\mddefault}{\itdefault}{\color[rgb]{1,0,0}$a_3$}%
}}}}
\put(2956,-11049){\makebox(0,0)[lb]{\smash{{\SetFigFont{20}{24.0}{\rmdefault}{\mddefault}{\itdefault}{\color[rgb]{1,0,0}$a_2$}%
}}}}
\put(2116,-10434){\makebox(0,0)[lb]{\smash{{\SetFigFont{20}{24.0}{\rmdefault}{\mddefault}{\itdefault}{\color[rgb]{1,0,0}$a_1$}%
}}}}
\put(2589,-4591){\makebox(0,0)[lb]{\smash{{\SetFigFont{20}{24.0}{\rmdefault}{\mddefault}{\itdefault}{\color[rgb]{1,0,0}$x+a_2+b_2$}%
}}}}
\put(6459,-6376){\rotatebox{300.0}{\makebox(0,0)[lb]{\smash{{\SetFigFont{20}{24.0}{\rmdefault}{\mddefault}{\itdefault}{\color[rgb]{1,0,0}$2y+2a_1+2a_3+2b_1$}%
}}}}}
\put(6984,-14169){\rotatebox{60.0}{\makebox(0,0)[lb]{\smash{{\SetFigFont{20}{24.0}{\rmdefault}{\mddefault}{\itdefault}{\color[rgb]{1,0,0}$2y+z+2a_2+2b_2$}%
}}}}}
\put(2865,-16169){\makebox(0,0)[lb]{\smash{{\SetFigFont{20}{24.0}{\rmdefault}{\mddefault}{\itdefault}{\color[rgb]{1,0,0}$x+a_1+a_3+b_1$}%
}}}}
\put(4620,-4573){\makebox(0,0)[lb]{\smash{{\SetFigFont{20}{24.0}{\rmdefault}{\mddefault}{\itdefault}{\color[rgb]{1,0,0}$z$}%
}}}}
\put(15242,-11563){\makebox(0,0)[lb]{\smash{{\SetFigFont{20}{24.0}{\rmdefault}{\mddefault}{\itdefault}{\color[rgb]{1,0,0}$z+b_1+b_2$}%
}}}}
\put(11064,-10006){\makebox(0,0)[lb]{\smash{{\SetFigFont{20}{24.0}{\rmdefault}{\mddefault}{\itdefault}{\color[rgb]{1,0,0}$a_1+a_2+a_3$}%
}}}}
\put(11174,-4592){\makebox(0,0)[lb]{\smash{{\SetFigFont{20}{24.0}{\rmdefault}{\mddefault}{\itdefault}{\color[rgb]{1,0,0}$x+a_2+b_2$}%
}}}}
\put(15044,-6377){\rotatebox{300.0}{\makebox(0,0)[lb]{\smash{{\SetFigFont{20}{24.0}{\rmdefault}{\mddefault}{\itdefault}{\color[rgb]{1,0,0}$2y+2a_1+2a_3+2b_1$}%
}}}}}
\end{picture}}
\caption{(a) Obtaining a $H^{(1)}$-type region in  \cite[Theorem 2.4]{Halfhex3} by removing forced lozenges. (b) Obtaining a halved hexagon with a semi-triangle removed in \cite[Proposition 2.1]{Ciucu1} by removing forced lozenges.}\label{reflectforced2}
\end{figure}

In the view of this, one can view Theorem \ref{halffactor1} as a common generalization of main results in \cite{Ranjan1,Halfhex1,Halfhex2,Halfhex3} as follows.  

We will show that one can use Theorem \ref{halffactor1} to obtain known enumerations of halved hexagons with ferns removed. The idea is picking suitable index sets $U,D,U',D'$  such that the regions in the numerator on the left-hand side of (\ref{halfeq1}) is  what we want to enumerate, and that the region in the denominator becomes a known region after removing certain forced lozenges.

Let us show in detail the implication to the main result in \cite{Halfhex3} (Theorem 2.4). We now pick the index sets $U,D,U',D'$ such that:
\begin{enumerate}
\item $k=2$,
\item the fern $F_1$ corresponding to the cluster $C_1$ contains $m$ triangles of side-lengths $a_1,a_2,\dots,a_m$ from left to right, starting by an up-pointing triangles,
\item the fern $F_2$ corresponding to the cluster $C_2$ contains $n+1$ triangles of side-length $z,b_1,b_2,\dots,b_n$ from right to left, starting by a down-pointing triangles,
\item the fern $F'_1$ corresponding to $C'_1$ consists of a single up-pointing triangle, 
\item the fern $F'_2$ corresponding to $C'_2$ consists of all a single down-pointing triangle.
\end{enumerate}
  This way, after removing forced lozenges, the region $F_{x,y}(U;D;\emptyset)=F_{x,y}(C_1,C_2|\ x+y)$ becomes the halved hexagon with two ferns removed $H^{(1)}_{x,y,z}(a_1,a_2,\dots,a_m; b_1,b_2,\dots,b_n)$ in \cite[Theorem 2.4]{Halfhex3}; while the region $F_{x,y}(U';D';\emptyset)=F_{x,y}(C'_1,C'_2|\ x+y)$ becomes a halved hexagon with a semi-triangle removed from the west side, that is previously enumerated in \cite[Proposition 2.1]{Ciucu1} (see illustration in Figure \ref{reflectforced2}; the ferns are indicated by chains of triangles with bold boundaries; the regions in \cite[Theorem 2.4]{Halfhex3} and in \cite[Proposition 2.1]{Ciucu1} are illustrated by the shaded ones in figures (a) and (b), respectively). Thus, our Theorem \ref{halffactor1} gives a simple product formula for the number of tilings of $H^{(1)}_{x,y,z}(a_1,a_2,\dots,a_m; b_1,b_2,\dots,b_n)$ and implies Theorem 2.4 in \cite{Halfhex3}.

If we specialize further by setting either $F_1=F'_1=\emptyset$ or $F'_2=F_2$ consisting of a single down-pointing triangle of side-length $z$, we get back the main results in \cite{Halfhex1} and \cite{Halfhex2}, respectively. Moreover, if we specialize even further by setting $F_1=F'_1=\emptyset$ and $F_2$ consists of only two triangles, then we get the main result in Rohatgi's paper \cite{Ranjan1}.

\section{Proofs of the main theorems}

\begin{figure}\centering
\includegraphics[width=12cm]{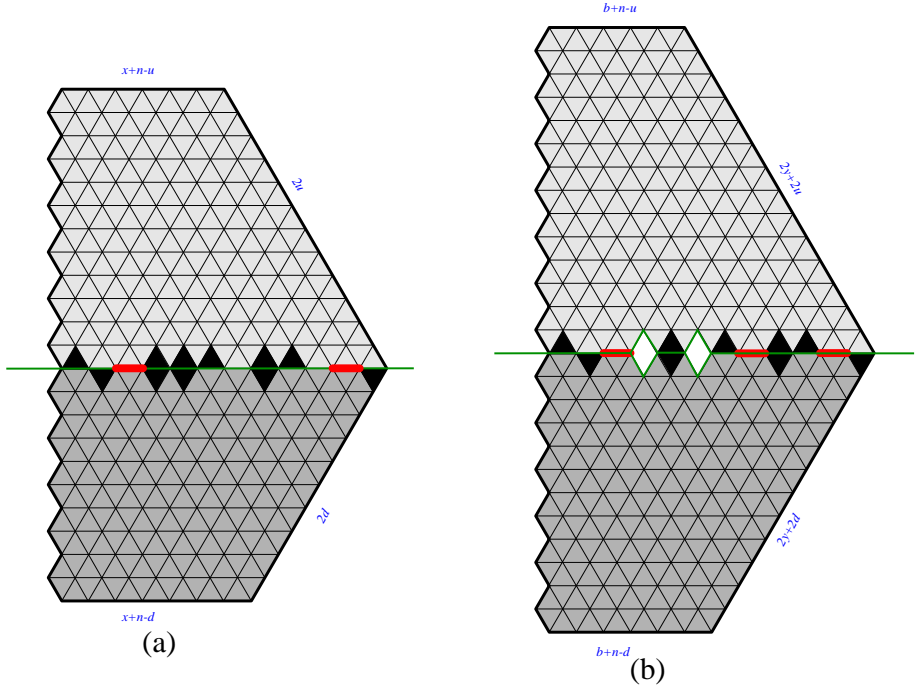}
\caption{Base cases: (a) $y=0$ and (b) $x=2b$}\label{basecase}
\end{figure}

\begin{figure}\centering
\includegraphics[width=10cm]{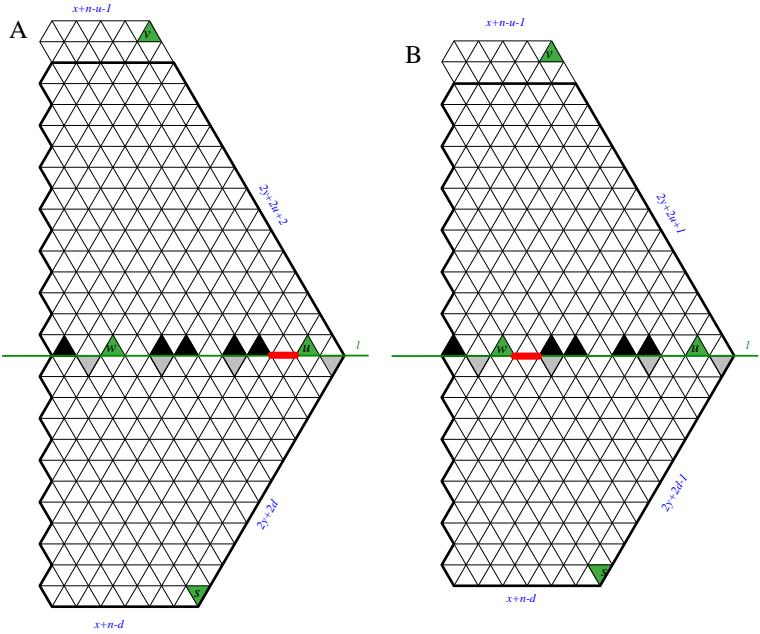}
\caption{How to apply Kuo condensation to a halved hexagon.}\label{halvedhexkuo1}
\end{figure}
\begin{figure}\centering
\includegraphics[width=9cm]{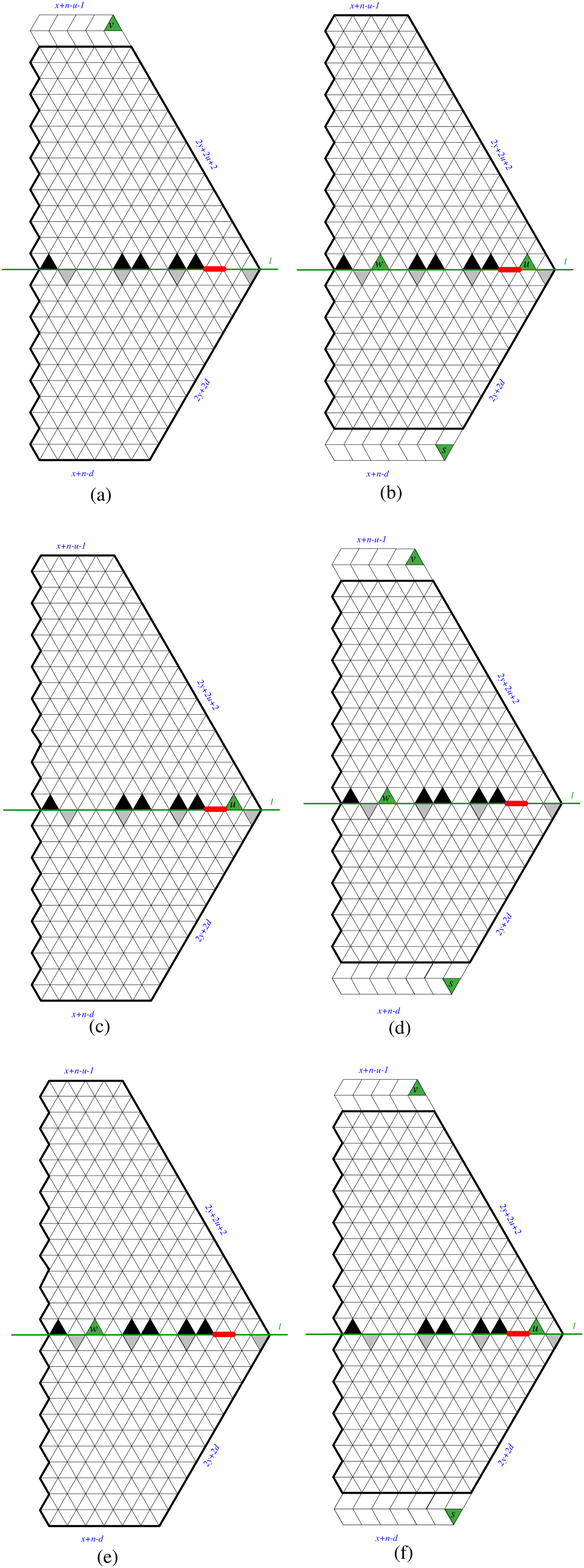}
\caption{Obtaining the recurrence for halved hexagons.}\label{halvedhexkuo2}
\end{figure}

\begin{proof}[Proof of Theorem \ref{halffactor1}]
We prove first the identity (\ref{halfeq1}) for the halved hexagon $F_{x,y}(U; D; B)$.  We prove by induction on $x+y$ with the base cases are $x=b$ and $y=0$.

If $y=0$, then we apply Region--splitting Lemma \ref{RS} to the region $R=F_{x,0}(U;D;B)$ in which the subregion $Q$ is its portion above the horizontal axis $l$. The subregion $Q$ is congruent with the quartered hexagon $L_{2u,x+n-u}(U)$ , and 
 its complement $R-Q$, after reflected over the horizontal axis $l$, is congruent with $L_{2u,b+n-d}(D)$ (see Figure \ref{basecase}(a)). In particular, we have
\[\M(F_{x,0}(U;D;B))=\M(L_{2u,x+n-u}(U))\M(L_{2u,b+n-d}(D)).\]
Similarly, we also get
\[\M(F_{x,0}(U';D';B))=\M(L_{2u',x+n-u'}(U'))\M(L_{2u',b+n-d'}(D')).\] Then (\ref{halfeq1}) follows directly from Lemma \ref{QAR}.

If $x=b$, we apply Region--splitting Lemma \ref{RS}, to  the region $R=F_{b,y}(U;D;B)$ in which the subregion $Q$ is the portion above the horizontal axis $l$ with the unit triangles at the positions in $(B\cup U\cup D)^c$ removed (we use the notation $S^{c}$ for the complement $[x+y+n]\setminus S$ of the index set $S\subseteq[x+y+n]$). $Q$ is congruent with the quartered hexagon $L_{2y+2u,b+n-u}((B\cup U\cup D)^c\cup U)$, and $R-Q$, after removing forced lozenges and reflected over $l$, is congruent with $L_{2y+2u,b+n-d}((B\cup U\cup D)^c\cup D$ (see Figure \ref{basecase}(b)). More precisely, we have
\[\M(F_{b,y}(U;D;B))=\M(L_{2y+2u,b+n-u}((B\cup U\cup D)^c\cup U))\M(L_{2y+2u,b+n-d}((B\cup U\cup D)^c\cup D)).\]
Process similarly for the region $R'=F_{b,y}(U';D';B)$, we get
\[\M(F_{b,y}(U';D';B))=\M(L_{2y+2u,b+n-u}((B\cup U'\cup D')^c\cup U'))\M(L_{2y+2u,b+n-d}((B\cup U'\cup D')^c\cup D')),\]
and (\ref{halfeq1}) follows again from Lemma \ref{QAR}, after performing a straight forward simplification.

\medskip

For the induction step, we assume that $x>b$, $y>0$ and that (\ref{halfeq1}) holds for any $F$-type regions whose sum of $x$- and $y$-parameters is strictly less than $x+y$. We will use Kuo condensation in Theorem \ref{kuothm} to obtain a recurrence for the left-hand side of (\ref{halfeq1}) and we show that the expression on the right-hand side satisfies the same recurrence, and then (\ref{halfeq1}) follows from the induction principle.

\medskip

We first apply Kuo condensation in Theorem \ref{kuothm} to the dual graph $G$ of the region $R$ obtained from $F_{x,y}(U;D;B)$ by adding to its top two layers of unit triangles as in Figure \ref{halvedhexkuo1}(a) (the region restricted by the bold contour indicate the region $F_{x,y}(U;D;B)$). The choice of the four vertices $u,v,w,s$  is shown by the corresponding unit triangles (i.e. the one with the same label) in the region $R$. In particular, we the $v$- and $s$-triangles are respectively the up-pointing  and down-pointing shaded unit triangles on the upper-right and lower-right corners of the region. The $u$- and $w$-triangles are placed on the horizontal axis $l$ at the first and the last positions  in $(U\cup D \cup B)^c$. Assume, in the rest of this proof, these two positions are $\alpha$ and $\beta$ ($\alpha< \beta$).

 Let us consider the region corresponding with the graph $G-\{v\}$. The removal of the $v$-triangle yields several forced lozenges on the top two layers of unit triangles. After removing these forced lozenges, we get back the region $F_{x,y}(U;D;B)$ (see Figure \ref{halvedhexkuo2}(a)). We note that in the unweighted case (as in the case of the region $F_{x,y}(U;D;B)$), the removal of forced lozenges dow \emph{not} change the tiling number of the region. This means that we have
\begin{equation}\label{halfeq1a}
\M(G-\{v\})=\M(F_{x,y}(U;D;B)).
\end{equation}
Considering forced lozenges as in Figures \ref{halvedhexkuo2}(b)--(f), we get five more identities
\begin{equation}\label{halfeq1b}
\M(G-\{v\})=\M(F_{x-1,y-1}(\alpha \beta U; D)),
\end{equation}
\begin{equation}\label{halfeq1c}
\M(G-\{v\})=\M(F_{x-1,y}(\beta U;D)),
\end{equation}
\begin{equation}\label{halfeq1d}
\M(G-\{v\})=\M(F_{x,y-1}(\alpha U;D)),
\end{equation}
\begin{equation}\label{halfeq1e}
\M(G-\{v\})=\M(F_{x-1,y}(\alpha U;D)),
\end{equation}
\begin{equation}\label{halfeq1f}
\M(G-\{v\})=\M(F_{x,y-1}(\beta U;D)).
\end{equation}
Here, for an index  set $S$, we use the shorthand notations $\alpha S, \beta S,\alpha\beta S$ for the unions $\{\alpha\} \cup S, \{\beta\} \cup S, \{\alpha,\beta \} \cup S$, respectively.
Plugging (\ref{halfeq1a})--(\ref{halfeq1f}) in the equation of Theorem \ref{kuothm}, we get a recurrence for the tiling number on the right-hand side of (\ref{halfeq1}):
\begin{align}\label{recurrence1}
\M(F_{x,y}(U;D);B)&\M(F_{x-1,y-1}(\alpha \beta U; D;B))=\M(F_{x-1,y}(\beta U;D))\M(F_{x,y-1}(\alpha U;D;B))\notag\\
&+\M(F_{x-1,y}(\alpha U;D;B))\M(F_{x,y-1}(\beta U;D;B))
\end{align}

\medskip

To complete the proof we will show that the expression on the right-hand side of (\ref{halfeq1}) also satisfies recurrence \ref{recurrence1}).
Equivalently, we need to verify that
\begin{align}\label{recurrence1b}
&\frac{\Delta(U)\Delta(D)}{\Delta(U')\Delta(D')}\frac{\Hf_2(2u'+2y+1)\Hf_2(2d'+2y+1)}{\Hf_2(2u+2y+1)\Hf_2(2d+2y+1)}\M(F_{x,y}(U';D';B))\notag\\
&\times\frac{\Delta(\alpha \beta U)\Delta(D)}{\Delta(\alpha \beta U')\Delta(D')}\frac{\Hf_2(2u'+2y+3)\Hf_2(2d'+2y-1)}{\Hf_2(2u+2y+3)\Hf_2(2d+2y-1)}\M(F_{x-1,y-1}(\alpha \beta U'; D';B))\notag\\
&=\frac{\Delta(\beta U)\Delta(D)}{\Delta(\beta U')\Delta(D')}\frac{\Hf_2(2u'+2y+3)\Hf_2(2d'+2y+1)}{\Hf_2(2u+2y+3)\Hf_2(2d+2y+1)}\M(F_{x-1,y}(\beta U';D'; B))\notag\\
&\times\frac{\Delta(\alpha U)\Delta(D)}{\Delta(\alpha U')\Delta(D')}\frac{\Hf_2(2u'+2y+1)\Hf_2(2d'+2y-1)}{\Hf_2(2u+2y+1)\Hf_2(2d+2y-1)}\M(F_{x,y-1}(\alpha U';D'; B))\notag\\
&+\frac{\Delta(\alpha U)\Delta(D)}{\Delta(\alpha U')\Delta(D')}\frac{\Hf_2(2u'+2y+3)\Hf_2(2d'+2y+1)}{\Hf_2(2u+2y+3)\Hf_2(2d+2y+1)}\M(F_{x-1,y}(\alpha U';D';B))\notag\\
&\times \frac{\Delta(\beta U)\Delta(D)}{\Delta(\beta U')\Delta(D')}\frac{\Hf_2(2u'+2y+1)\Hf_2(2d'+2y-1)}{\Hf_2(2u+2y+1)\Hf_2(2d+2y-1)}\M(F_{x,y-1}(\beta U';D';B)).
\end{align}
Recall that, for an ordered set $S=\{s_1<s_2<\cdots<s_k\}$, we define $\Delta(S):=\prod_{1\leq i <j \leq k}(s_j^2-s_i^2)$.

Dividing both sides by $\frac{\Delta(D)}{\Delta(D')}\frac{\Hf_2(2u'+2y+1)\Hf_2(2d'+2y+1)}{\Hf_2(2u+2y+1)\Hf_2(2d+2y+1)}$, we simplify the above equation to
\begin{align}\label{recurrence1c}
&\frac{\Delta(U)}{\Delta(U')}\M(F_{x,y}(U';D';B))\frac{\Delta(\alpha \beta U)}{\Delta(\alpha \beta U')}\M(F_{x-1,y-1}(\alpha \beta U'; D';B))\notag\\
&=\frac{\Delta(\beta U)}{\Delta(\beta U')}\M(F_{x-1,y}(\beta U';D'; B))\frac{\Delta(\alpha U)}{\Delta(\alpha U')}\M(F_{x,y-1}(\alpha U';D'; B))\notag\\
&+\frac{\Delta(\alpha U)}{\Delta(\alpha U')}\M(F_{x-1,y}(\alpha U';D';B)) \frac{\Delta(\beta U)}{\Delta(\beta U')}\M(F_{x,y-1}(\beta U';D';B)).
\end{align}
Next, we divide both sides of (\ref{recurrence1c}) by  $\frac{\Delta(U)\Delta(\alpha\beta U)}{\Delta(U')\Delta(\alpha\beta U')}$ and get
\begin{align}\label{recurrence1d}
\M(F_{x,y}(U';D';B))&\M(F_{x-1,y-1}(\alpha \beta U'; D';B))\notag\\
&=\frac{\Delta(\beta U)\Delta(\alpha U)}{\Delta(U)\Delta(\alpha\beta U)}\frac{\Delta(U')\Delta(\alpha\beta U')}{\Delta(\beta U')\Delta(\alpha U')}\M(F_{x-1,y}(\beta U';D'; B))\M(F_{x,y-1}(\alpha U';D'; B))\notag\\
&+\frac{\Delta(\alpha U)\Delta(\beta U)}{\Delta(U)\Delta(\alpha\beta U)}\frac{\Delta(U')\Delta(\alpha\beta U')}{\Delta(\alpha U')\Delta(\beta U')}\M(F_{x-1,y}(\alpha U';D';B)) \M(F_{x,y-1}(\beta U';D';B)).
\end{align}
We claim that
\begin{claim}
\begin{equation}\label{ratio1}
\frac{\Delta(\beta U)\Delta(\alpha U)}{\Delta(U)\Delta(\alpha\beta U)}\frac{\Delta(U')\Delta(\alpha\beta U')}{\Delta(\beta U')\Delta(\alpha U')}=1
\end{equation}\
and
\begin{equation}\label{ratio2}
\frac{\Delta(\alpha U)\Delta(\beta U)}{\Delta(U)\Delta(\alpha\beta U)}\frac{\Delta(U')\Delta(\alpha\beta U')}{\Delta(\alpha U')\Delta(\beta U')}=1.
\end{equation}
\end{claim}
\begin{proof}[Proof of the claim]
Let us prove first (\ref{ratio1}). Simplifying the first fraction on the left-hand side by canceling out the common terms in the $\Delta$-products, we get
\begin{align}
\frac{\Delta(\beta U)\Delta(\alpha U)}{\Delta(U)\Delta(\alpha\beta U)}&=\frac{\prod_{i=1}^{i}|\beta-s_i|}{|\beta-\alpha|\prod_{i=1}^{u}|\beta-s_i|}\notag\\
&=\frac{1}{\beta-\alpha}.
\end{align}
Similarly, the second fraction can be simplified to just $(\beta-\alpha)$, and (\ref{ratio1}) follows.

By interchanging the roles of $\alpha$ and $\beta$ in (\ref{ratio1}), one obtain (\ref{ratio2}). This finishes the proof of the claim.
\end{proof}

By the claim, (\ref{recurrence1c}) now becomes
\begin{align}\label{recurrence1e}
\M(F_{x,y}(U';D';B))\M(F_{x-1,y-1}&(\alpha \beta U'; D';B))=\M(F_{x-1,y}(\beta U';D'; B))\M(F_{x,y-1}(\alpha U';D'; B))\notag\\
&+\M(F_{x-1,y}(\alpha U';D';B)) \M(F_{x,y-1}(\beta U';D';B)).
\end{align}
However, this follows directly from the application of recurrence (\ref{recurrence1}) to the region $F_{x,y}(U';D';B)$. This finishes the proof of (\ref{halfeq1}).

\medskip

The proof of (\ref{halfeq2}) is similar to that of (\ref{halfeq1}). We also prove by induction on $x+y$, with the base cases are still the situations when $x=b$ and $y=0$.

If $y=0$, we get from Region-splitting Lemma \ref{RS} by diving the regions along the horizontal axis $l$ into two quartered hexagons:
\[\M(\overline{F}_{x,0}(U;D;B))=\M(L_{2u-1,x+n-u}(U))\M(L_{2u-1,x+n-d}(D))\]
and
\[\M(\overline{F}_{x,0}(U';D';B))=\M(L_{2u'-1,x+n-u'}(U'))\M(L_{2u'-1,x+n-d'}(D')).\]  Then (\ref{halfeq2}) follows from Lemma \ref{QAR}.

For the case $x=b$, we also have from Region-splitting Lemma \ref{RS}:
\[\M(\overline{F}_{b,y}(U;D;B))=\M(L_{2y+2u-1,b+n-u}((B\cup U\cup D)^c \cup U))\M(L_{2y+2u-1,b+n-d}((B\cup U\cup D)^c \cup D)\]
and
\[\M(\overline{F}_{b,y}(U';D';B))=\M(L_{2y+2u-1,b+n-u}((B\cup U'\cup D')^c \cup U'))\M(L_{2y+2u-1,b+n-d}((B\cup U'\cup D')^c \cup D')).\]
Again, (\ref{halfeq2}) follows from Lemma \ref{QAR}.

The induction step is completely analogous to that in the proof of (\ref{halfeq1}), we now use Kuo condensation as in Figure \ref{halvedhexkuo1}(b) to get the recurrence for the left-hand side of (\ref{halfeq2}):
\begin{align}\label{recurrence2}
\M&(\overline{F}_{x,y}(U;D;B))\M(\overline{F}_{x-1,y-1}(\alpha \beta U; D;B))=\M(\overline{F}_{x-1,y}(\beta U;D;B))\M(\overline{F}_{x,y-1}(\alpha U;D;B))\notag\\
&+\M(\overline{F}_{x-1,y}(\alpha U;D;B))\M(\overline{F}_{x,y-1}(\beta U;D;B)).
\end{align}
Working similarly to the proof of (\ref{halfeq1}) above, one can verify that the expression on the right-hand side of (\ref{halfeq2}) also satisfies recurrence (\ref{recurrence2}). Then (\ref{halfeq2}) follows by the induction principle, and this finishes the proof for (\ref{halfeq2}).
\end{proof}

The proof for Theorem \ref{halffactor2} is essentially the same as that of Theorem \ref{halffactor1} above. We apply Kuo condensation exactly the same as in Figure \ref{halvedhexkuo1} to obtain same recurrences as that in the proof of Theorem \ref{halffactor1}. We note that, even though, lozenges of the $W$- and $\overline{W}$-type regions may carry weights, our forced lozenges all have weight 1. This means that our arguments in the proof of Theorem \ref{halffactor1} still work well in proving Theorem \ref{halffactor2}. We omit this proof here.

\medskip

We now are ready to proof Theorem \ref{reflectfactor}.

\begin{figure}\centering
\includegraphics[width=13cm]{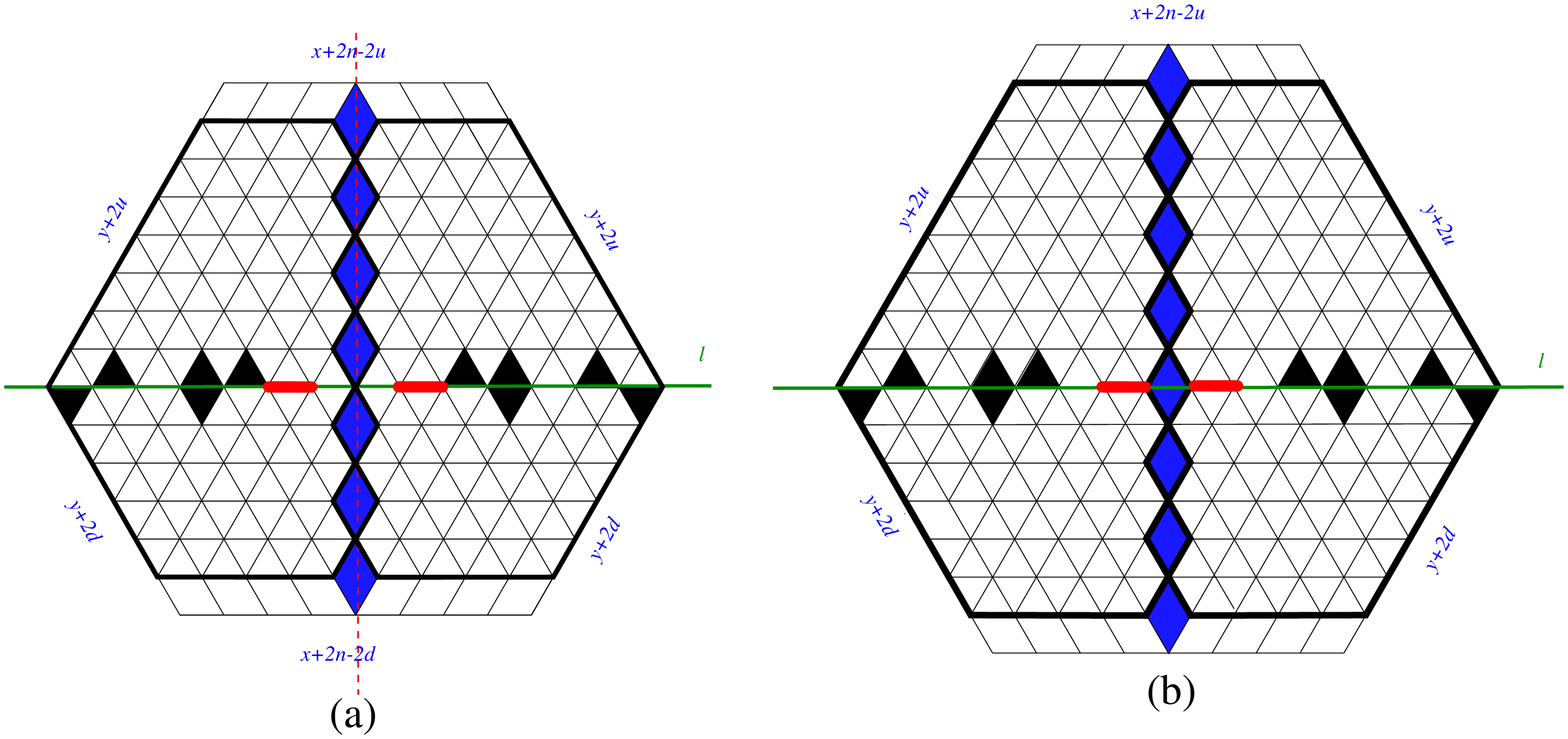}
\caption{Illustrating the proof of Theorem \ref{reflectfactor} in the case (a) $y$ is even and (b) $y$ is odd.}\label{Reflectdent2}
\end{figure}

\begin{proof}[Proof of Theorem \ref{reflectfactor}]
One readily see that, any reflectively symmetric tiling of the region $RS_{x,y}(U;D;B)$ must contain all vertical lozenges along the vertical symmetry axis. The removal of these lozenges divides the region into two congruent regions that are images of each other through the reflective symmetry (see Figure \ref{Reflectdent2}). This means that reflectively symmetric tilings of $RS_{x,y}(U;D;B)$ are in bijection with the tilings of each of these two regions. Remove forced lozenges from the top and bottom of each of these regions, we get back exactly the halved hexagon $\overline{F}_{x,\frac{y}{2}}(U;D;B)$ when $y$ is even and the halved hexagon $F_{x,\frac{y-1}{2}}(U;D;B)$ when $y$ is odd. Then our theorem follows from Theorem \ref{halffactor1}.
\end{proof}

\section{Concluding Remarks}
One readily sees that the number of tilings of $L_{2k,n}(a_1,\dots,a_k)$ can be written in terms of \emph{symplectic characters} (irreducible characters of \emph{symplectic group} $Sp_{2n}(\mathbb{C})$: $sp_{\lambda}(1,1,\dots,1)$, where $\lambda$ is the partition $(a_k-k, a_{k-1}-k+1,\dots,a_1-1)$ (see, e.g. \cite{Fischer}). One can also write the number of tilings of $F_{x,y}(U;D;B)$ as the sum
\[\M(F_{x,y}(U;D;B))=\sum_{|S|=y}sp_{\lambda(U\cup S)}(1^{u+y})sp_{\lambda(D\cup S)}(1^{d+y}),\]
where the sum is taken over all $y$-subsets of the complement of $U\cup D\cup B$. Indeed, each lozenge tiling of $F_{x,y}(U;D;B)$ contain exactly $y$ vertical lozenges at the positions in the complement of $U\cup D\cup B$ along the horizontal axis. Removing these $y$ vertical lozenges, we partition the tiling into tilings of two disjoint quartered hexagons of type $L_{2k,n}$. This way, the specialization of identity (\ref{halfeq1}), when $|U|=|U'|$ and $|D|=|D'|$, can be rewritten as:
\begin{equation}\label{halfeq1refine}
\frac{\sum_{|S|=y}sp_{\lambda(U\cup S)}(1^{u+y})sp_{\lambda(D\cup S)}(1^{d+y})}{\sum_{|S|=y}sp_{\lambda(U'\cup S)}(1^{u+y})sp_{\lambda(D'\cup S)}(1^{d+y})}=\frac{sp_{\lambda(U)}(1^{u})sp_{\lambda(D)}(1^{d})}{sp_{\lambda(U')}(1^{u})sp_{\lambda(D')}(1^{d})}.
\end{equation}
It would be interesting to see if there exists an identity of symplectic characters behind this.

We consider the special case of Theorem \ref{halffactor1} when $|U|=|U'|$ and $|D|=|D'|$. The identity \ref{halfeq1} can be simplified as
\begin{align}
\frac{\M(F_{x,y}(U;D;B))}{\M(F_{x,y}(U;D;B))}=\frac{\displaystyle\prod_{1\leq i <j\leq u}(s^2_j-s^2_i)\displaystyle\prod_{1\leq i <j\leq d}(t^2_j-t^2_i)}{\displaystyle\prod_{1\leq i <j\leq u}(s'^2_j-s'^2_i)\displaystyle\prod_{1\leq i <j\leq d}(t'^2_j-t'^2_i)}.
\end{align}
Comparing with Lemma \ref{QAR}, we can rewrite the above identity in terms of tilings numbers of halved hexagons and quartered hexagons
\begin{align}\label{geoeq1}
\frac{\M(F_{x,y}(U;D;B))}{\M(F_{x,y}(U;D;B))}=\frac{\M(L_{2u,x+y+n-u}(U))\M(L_{2d,x+y+n-d}(D))}{\M(L_{2u,x+y+n-u}(U'))\M(L_{2d,x+y+n-d}(D'))}.
\end{align}
It is interesting to find a combinatorial explanation for this geometric interpretation. 

Moreover, by Region-Splitting Lemma \ref{RS}, we can show that
\begin{equation}
\M(F_{x+y,0}(U;D;B))=\M(L_{2u,x+y+n-u}(U)\M(L_{2d,x+y+n-d}(D))
\end{equation}
and
\begin{equation}
\M(F_{x+y,0}(U';D';B))=\M(L_{2u,x+y+n-u}(U')\M(L_{2d,x+y+n-d}(D')).
\end{equation}
Then (\ref{geoeq1}) can be written as
\begin{align}\label{geoeq2}
\M(F_{x,y}(U;D;B))\M(F_{x+y,0}(U';D';B))=\M(F_{x,y}(U;D;B))\M(F_{x+y,0}(U;D;B)).
\end{align}
The both sides of the above identity count pairs of tilings of certain halved hexagons. It would be interesting to find a bijective proof for identity (\ref{geoeq2}).

\section*{Acknowledgement}
The author would like to thank Ranjan Rohatgi for fruitful discussions. Part of the main result  in this paper was initiated from the joint project between the author and him on shuffling theorems of doubly--dented hexagons in \cite{shuffling}.

\end{document}